\documentclass[12pt,oneside,reqno]{amsart}

\usepackage{amsmath,amssymb,amsthm,textcomp}
\usepackage{amsfonts,graphicx}
\usepackage[mathscr]{eucal}
\usepackage{color}
\usepackage{csquotes}
\usepackage[backend=bibtex,%
firstinits=true,%
doi=false,%
isbn=true,%
url=false,%
maxnames=99]{biblatex}%

\AtEveryBibitem{\clearfield{issn}}
\AtEveryCitekey{\clearfield{issn}}
\numberwithin{equation}{section}
\DeclareNameAlias{sortname}{last-first}
\theoremstyle{definition}
\usepackage{mathtools}
\numberwithin{equation}{section}

\newcommand{\ncom}{\newcommand}

\ncom{\beq}{\begin{equation}}
	\ncom{\eeq}{\end{equation}}
\ncom{\bea}{\begin{eqnarray*}}
	\ncom{\eea}{\end{eqnarray*}}
\ncom{\beqa}{\begin{eqnarray}}
	\ncom{\eeqa}{\end{eqnarray}}
\ncom{\nno}{\nonumber}
\ncom{\non}{\nonumber}
\ncom{\ds}{\displaystyle}
\ncom{\half}{\frac{1}{2}}
\ncom{\mbx}{\makebox{.25cm}}
\ncom{\hs}{\mbox{\hspace{.25cm}}}
\ncom{\rar}{\rightarrow}
\ncom{\Rar}{\Rightarrow}
\ncom{\noin}{\noindent}
\ncom{\bc}{\begin{center}}
	\ncom{\ec}{\end{center}}
\ncom{\sz}{\scriptsize}
\ncom{\rf}{\ref}
\ncom{\s}{\sqrt{2}}
\ncom{\sgm}{\sigma}
\ncom{\Sgm}{\Sigma}
\ncom{\psgm}{\sigma^{\prime}}
\ncom{\dt}{\delta}
\ncom{\Dt}{\Delta}
\ncom{\lmd}{\lambda}
\ncom{\Lmd}{\Lambda}
\ncom{\Th}{\Theta}
\ncom{\e}{\eta}
\ncom{\eps}{\varepsilon}
\ncom{\pcc}{\stackrel{P}{>}}
\ncom{\lp}{\stackrel{L_{p}}{>}}
\ncom{\dist}{{\rm\,dist}}
\ncom{\sspan}{{\rm\,span}}
\ncom{\re}{{\rm Re\,}}
\ncom{\im}{{\rm Im\,}}
\ncom{\sgn}{{\rm sgn\,}}
\ncom{\ba}{\begin{array}}
	\ncom{\ea}{\end{array}}
\ncom{\hone}{\mbox{\hspace{1em}}}
\ncom{\htwo}{\mbox{\hspace{2em}}}
\ncom{\hthree}{\mbox{\hspace{3em}}}
\ncom{\hfour}{\mbox{\hspace{4em}}}
\ncom{\vone}{\vskip 2ex}
\ncom{\vtwo}{\vskip 4ex}
\ncom{\vonee}{\vskip 1.5ex}
\ncom{\vthree}{\vskip 6ex}
\ncom{\vfour}{\vspace*{8ex}}
\ncom{\norm}{\|\;\;\|}
\ncom{\integ}[4]{\int_{#1}^{#2}\,{#3}\,d{#4}}
\ncom{\vspan}[1]{{{\rm\,span}\{ #1 \}}}
\ncom{\dm}[1]{ {\displaystyle{#1} } }
\ncom{\ri}[1]{{#1} \index{#1}}

\newtheorem{theorem}{\bf Theorem}[section]
\newtheorem{remark}{\bf Remark}[section]
\newtheorem{proposition}{Proposition}[section]
\newtheorem{lemma}{Lemma}[section]

\newtheoremstyle
{remarkstyle}
{}
{11pt}
{}
{}
{\bfseries}
{:}
{     }
{\thmname{#1} \thmnumber{#2} }

\theoremstyle{remarkstyle}

\def\eps{\varepsilon}

\date{\today}
\begin{document}
\title{On Elephant Random Walk with Delayed Amnesia}
\author[Manisha Dhillon]{Manisha Dhillon}
\address{Manisha Dhillon, Department of Mathematics, Indian Institute of Technology Bhilai, Durg 491002, India.}
\email{manishadh@iitbhilai.ac.in}
\author[Shyan Ghosh]{Shyan Ghosh}
\address{Shyan Ghosh, Department of Mathematics, Indian Institute of Technology Bhilai, Durg 491002, India.}
\email{shyanghosh@iitbhilai.ac.in}
\author[Kuldeep Kumar Kataria]{Kuldeep Kumar Kataria}
\address{Kuldeep Kumar Kataria, Department of Mathematics, Indian Institute of Technology Bhilai, Durg 491002, India.}
\email{kuldeepk@iitbhilai.ac.in}

\begin{abstract}
In this paper, we introduce a modified elephant random walk that exhibits a transition from a uniform memory mechanism to a selective amnesic memory mechanism. Using a vector martingale approach, we study the asymptotic behaviour of the walk across different parameter regimes. In the diffusive and critical regimes, we establish almost sure convergence results, laws of iterated logarithm, asymptotic normality of the walk, and the growth rate of mean square displacement. In the superdiffusive regime, we prove an almost sure convergence result and obtain the corresponding mean square displacement rate for the walk. Also, we study some almost sure convergence results for its center of mass. Later, we extend the model by incorporating random step sizes and obtain asymptotic results for it.
\end{abstract}
\maketitle

\section{Introduction}
A simple random walk is one of the most widely studied stochastic models because of its mathematical simplicity and broad applicability. In a simple random walk, the successive increments are independent and identically distributed (iid) that lead to its Markovian structure, and so it does not retain any memory of its past. Therefore, such models typically exhibit diffusive behaviour. However, many real world phenomena show anomalous diffusion and long-range dependence features that cannot be captured by the simple random walk. This limitation has motivated the development of random walk models that incorporate memory effects. Sch\"utz and Trimper (2004) introduced elephant random walk (ERW) to study the long term memory effect on asymptotic behaviour of non-Markovian processes. The terminology `elephant' originates from the popular belief that elephants possess long memory. Unlike the simple random walk, the ERW evolves according to its entire past trajectory, and therefore exhibits strong dependence among increments. Depending on the value of memory parameter, the model may exhibit diffusive, critical, or superdiffusive behaviour. Thus, ERW is one of the most fundamental models exhibiting anomalous diffusion. Phenomena exhibiting anomalous diffusion can be observed across a wide range of physical, biological, and social systems. For example, it dictates human travel patterns (see Brockmann \textit{et al.} (2006)), plasma membrane ion channels (see Weigel \textit{et al.} (2011)). This behavior is also found in cell nucleus telomeres (see Bronstein \textit{et al.} (2009)), heartbeat intervals and DNA sequences (see Buldyrev et al. (1994)), diffusion in polymer networks (see Wong \textit{et al.} (2004)), \textit{etc.}

The standard ERW is one dimensional discrete time random walk on $\mathbb{Z}$. The walker starts from the origin and moves by one unit at each time step. Initially, the walker jumps to the right with probability $q\in[0,1]$ and to the left with probability $1-q$. Subsequently, for $n\ge 2$, one of the previous steps is selected uniformly at random from $\{1,2,\dots,n-1\}$. The walker then repeats the chosen step with probability $p\in[0,1]$ or moves in the opposite direction with probability $1-p$. The mathematical formulation is as follows:

Let $\{X_n\}_{n\ge1}$ be the sequence of increments. Then,
\begin{equation*}
	X_{1}=\begin{cases}
		+1\ \text{with probability}\ q,\\
		-1\ \text{with probability}\ 1-q,
	\end{cases}
\end{equation*}
and for every $n\ge1$,
\begin{equation*}
	X_{n+1}=\begin{cases}
		+X_{U(n)}\ \text{with probability}\ p,\\
		-X_{U(n)}\ \text{with probability}\ 1-p,
	\end{cases}
\end{equation*}
where $U(n)\sim$ Uniform$\{1,2,\dots,n\}$. The position of the walker after $n$ steps is given by
$S_n=X_1+X_2+\dots+ X_n$ such that $S_0=0$. As each future step depends on the entire history of the process, the ERW provides a natural framework for modelling systems with memory.

Over the past two decades, the ERW and its extensions have gained considerable attention of the researchers. Different probabilistic techniques such as martingale approach, stochastic approximation method, branching structures and P\'olya-type urn representations are used to study its asymptotic behaviour. Baur and Bertoin (2016) established connections between the ERW and P\'olya-type urn models, while Bercu (2018) obtained several asymptotic results including laws of large numbers, quadratic strong laws, laws of iterated logarithm and asymptotic normality using martingale techniques. Coletti \textit{et al.} (2017a, 2017b) derived limit theorems and invariance principles for the ERW. Several extensions of the ERW have been introduced including the ERW with stops (see Kumar \textit{et al.} (2010), Bercu (2022)), ERW with restricted memory (see Gut and Stadtm\"uller (2021a, 2021b)), ERW with random step sizes (see Dedecker \textit{et al.} (2023), Fan and Shao (2024)), and references therein. However, the ERW in multidimensional setting have been studied by Bercu and Laulin (2019), Bercu (2025), while its properties such as recurrence and transience are discussed in Curien and Laulin (2024) and Qin (2025). For other extensions of the ERW, we refer the reader to  Gangopadhyay and Maulik (2022), Maulik \textit{et al.} (2025), Roy \textit{et al.} (2025), Podder and Roy (2026), Peres and Qin (2026), Dhillon and Kataria (2026), and references therein.

Gut and Stadm\"uller (2021a, 2021b) extended the study of ERW by considering situations in which the walker has only a restricted memory, while Laulin (2022) introduced an ERW with a smooth amnesia mechanism. Also, Bertenghi and Laulin (2025) introduced a variation of the step-reinforced random walk with amnesic memory, and Majumdar and Maulik (2025) studied step reinforced random walk with regularly varying memory. Motivated by the fact that memory mechanisms in many real world systems may evolve over time, we introduce a modified elephant random walk whose dependence mechanism over past steps changes after a prescribed time instant.

In the standard ERW, every past step remains equally likely to be recalled throughout the entire evolution of the process. However, in many practical situations, the influence of distant past events gradually weakens and memory retrieval becomes more selective as time progresses. To capture this feature, we consider a random walk, namely, the ERW with delayed amnesia that initially evolves according to the standard ERW dynamics and subsequently switches to an amnesic memory regime. More precisely, up to a fixed time $(a-1)$, the walker follows the usual uniform memory rule of the standard ERW. After this transition point, the walker recalls past steps according to a non uniform probability distribution involving Gamma functions, which increasingly favours more recent steps. Therefore, the ERW with delayed amnesia exhibits a deterministic transition from uniform memory mechanism to a progressively selective memory mechanism. We study several asymptotic results of the walk and determine how the change in memory mechanism affects the long-term behaviour of the process.

Although our model is inspired by the amnesic ERW model introduced by Laulin (2022), the two memory mechanisms are fundamentally different. In Laulin (2022), the amnesic effect is present throughout the evolution of the walk. In contrast, our model begins with the uniform memory mechanism of the standard ERW and then introduces amnesia after the threshold parameter $a$. Thus, the parameter $a$ marks the onset of memory loss and governs the transition from a uniform memory regime to an amnesic one. A potential application of the proposed process is in the stochastic modelling of autonomous navigation strategies for planetary exploration missions. Such missions are often designed with a predetermined exploration phase followed by a targeted scientific operation phase. During the initial exploration phase, the rover systematically surveys an unknown terrain to construct a global map. During this stage, every previously traversed path is assumed to be equally likely to influence its subsequent movement, reflecting the need for broad and unbiased exploration. Once the predetermined exploration phase is completed, the mission shifts to localized scientific investigation. At this stage, recent navigation history becomes more informative than distant history, as the rover must adapt to the local terrain and current operational conditions. This planned change in navigation strategy can be modelled by the ERW with delayed amnesia, in which the walker follows the standard ERW with uniform memory up to a fixed time $a-1$ and thereafter transitions to an amnesic memory mechanism that increasingly favours more recent steps. 

%
The paper is organized as follows:

In Section \ref{prel}, we set some notations and collect some preliminary results that will be used later. In Section \ref{interm}, we introduce the ERW with delayed amnesia and construct a vector martingale which will be used to obtain the limiting results. Then, we study the asymptotic results for ERW with delayed amnesia in different regimes. In the diffusive and critical regimes, we establish several almost sure convergence results for the walk, including the law of iterated logarithm, the asymptotic normality and the growth rate of mean square displacement. In the superdiffusive regime, we prove both almost sure convergence and growth rate of mean square displacement for the walk. Moreover, we study the almost sure convergence results for the center of mass associated with the ERW with delayed amnesia. In the last section, we study an extension of our model by considering the random step sizes.

\section{Preliminaries}\label{prel}
The following notations will be used throughout the paper: Let $\mathbb{C}$, $\mathbb{R}$, $\mathbb{Z}$ and $\mathbb{N}$ be the set of complex numbers, real numbers, integers and positive integers, respectively. Also, let $\mathbb{R}^d=\mathbb{R}\times\mathbb{R}\times\cdots\times\mathbb{R}$ ($d$ copies) and $\mathbb{I}_G$ denote the indicator function of a set $G$. For two sequences $\{x_n\}_{n\ge 1}$ and $\{y_n\}_{n\ge 1}$, the notation $x_n \sim y_n$ indicates that $x_n/y_n \to 1$ as $n \to \infty$. However, for two positive sequences $\{x_n\}_{n\ge 1}$ and $\{y_n\}_{n\ge 1}$, the notations $x_n=O(y_n)$ indicates that $\{x_n/y_n\}_{n\ge 1}$ is bounded. For any matrix $A$, its transpose is denoted by $A^t$. Also, let $\lambda_{\max}(A)$ and $\|A\|$ be the maximum eigenvalue of $A$ and $\sqrt{\lambda_{\max}(A^tA)}$, respectively. Moreover, $\xrightarrow{p}$, $\xrightarrow{d}$ and $\overset{d}{=}$ denotes convergence in probability, convergence in distribution and equality in distribution, respectively. In expressions involving conditional expectations, the abbreviation a.s. for almost surely is omitted to avoid repetition.
 
Now, we collect some known results which will be used later.
\begin{theorem}[Duflo (1997), Theorem 1.3.15]\label{thm_LLN1}
Let $\{M_n, \mathcal{F}_n\}_{n\ge 0}$ be a square-integrable martingale with predictable quadratic variation $\{\langle M\rangle_n\}_{n\ge 0}$. Also, let $\langle M\rangle_\infty=\lim_{n\to\infty}\langle M\rangle_n$. Then,\vspace{.1cm}\\
\noindent{(i)}  $\lim_{n\to\infty}M_n=M_\infty$ a.s. on $\{\langle M\rangle_\infty < \infty\}$, where $M_\infty$ is a finite random variable, and\vspace{.1cm}\\		
\noindent{(ii)}  $\lim_{n\to\infty}M_n/\langle M\rangle_n=0$ a.s. on $\{\langle M\rangle_\infty = \infty\}$.
\end{theorem}

\begin{theorem}[Duflo (1997), Theorem 1.3.24 ]\label{thm_LLN2}
Let $\{\varepsilon_n\}_{n\ge 0}$ be a square-integrable random sequence which takes values in $\mathbb{C}^d$ and is adapted to some filtration $\{\mathcal{F}_n\}_{n\ge 0}$ such that $\mathbb{E}(\varepsilon_{n+1}|\mathcal{F}_n)=0$ and $\sup_{n\ge1}\mathbb{E}(\parallel\varepsilon_{n+1}\parallel^2|\mathcal{F}_n)\le C$ a.s., where $C$ is a finite random variable. Also, let $\{\Phi_n\}_{n\ge 0}$ be a sequence of complex $d$-dimensional random variables which is adapted to $\{\mathcal{F}_n\}_{n\ge 0}$, and $s_n=\sum_{k=0}^{n}\parallel\Phi_k\parallel^2$, $s_\infty=\lim_{n\to\infty}s_n$ and $M_n=\sum_{k=1}^{n}\langle \Phi_{k-1},\varepsilon_k\rangle$. Then, \vspace{.1cm}\\
\noindent{(i)}  $\{M_n\}_{n\ge 0}$ converges a.s. on the set $\{s_\infty <\infty\}$, and	\vspace{.1cm}\\
\noindent{(ii)} $\lim_{n\to\infty}M_n/s_{n-1}=0$ a.s. on the set $\{s_\infty =\infty\}$ (law of large numbers (LLN)).\vspace{.1cm}\\
Also, the rate of this LLN can be specified as follows: $|M_n|^2=O((\log s_{n-1})^{1+\gamma}s_{n-1})$ a.s. for all $\gamma>0$.

\end{theorem}
\subsection{Law of iterated logarithm for martingales}
Let $\{M_n,\mathcal{F}_n\}_{n\geq 1}$ be a martingale such that $\mathbb{E}(M_1)=0$. Also, let $\mathcal{F}_0=\{\emptyset,\Omega\}$,  $u_n=\sqrt{2\log\log t_n^2}$ and $t_n^2=\sum_{k=1}^{n}\mathbb{E}((\Delta M_k)^2\mid\mathcal{F}_{k-1})$, where $\Delta M_n=M_n-M_{n-1}$, $n\geq 1$
are martingale differences with $M_0=0$.

%

\begin{theorem}[Stout (1970), Theorem 1, Theorem 2]\label{stoutthm2}
If $\lim_{n\to\infty}t_n^2=\infty$ a.s. and $|\Delta M_n|
\leq
K_nt_n/u_n$, $n\geq 1$
for some $K_n$ that is $\mathcal{F}_{n-1}$-measurable such that $\lim_{n\to \infty}K_n= 0$ a.s. then
$\limsup_{n\to\infty}M_n/(t_nu_n)
= 1$ a.s.
\end{theorem}

\subsection{Central limit theorem (CLT)}
The following CLT for martingales will be used:	
\begin{theorem}[Duflo (1997), Corollary 2.1.10]\label{CLT}	Let $\{M_n,\mathcal{F}_n\}_{n\ge 0}$ be a real, square-integrable vector martingale whose predictable quadratic variation is $\{\langle M\rangle_n\}_{n\ge 0}$. Also, let the following assumptions hold for a deterministic real sequence $\{a_n\}_{n\ge 0}$ which increases to $\infty$:\\
\noindent {(i)} $a_n^{-1}\langle M\rangle_n\xrightarrow{p}\Lambda$, where $\Lambda$ is a deterministic symmetric positive semi-definite matrix, and\\
\noindent {(ii)} Lindeberg's condition is satisfied, that is, for all $\varepsilon>0$,
\begin{equation*}
a_n^{-1}\sum_{k=1}^{n}\mathbb{E}(\parallel M_k-M_{k-1}\parallel^2 \mathbb{I}_{\{\parallel M_k-M_{k-1}\parallel \ge \varepsilon \sqrt{a_n}\}}|\mathcal{F}_{k-1})\xrightarrow{p} 0.
\end{equation*}
Then, $a_n^{-1/2}M_n\xrightarrow{d} \mathcal{N}(0,\Lambda)$. 
\end{theorem}

The next result is a simplified version of Theorem 1 of Touati (1991) (see Bercu and Laulin (2021)).
\begin{theorem}\label{thm_asymnorm}
Let $\{M_n, \mathcal{F}_n\}_{n\ge1}$ be a locally square-integrable martingale in
$\mathbb{R}^{d}$ with predictable
quadratic variation $\{\langle M \rangle_n\}_{n\ge1}$.
Also, let $\{V_n\}_{n\ge1}$ be a sequence of non-random square matrices of order $d$ such that $\|V_n\|$ decreases to $0$ as $n\to \infty$. Assume that there exists a symmetric and positive semi-definite matrix $V$ such that
\begin{equation}\label{Touaticondi1}
V_n \langle M \rangle_n V_n^t
\xrightarrow{p} V.
\end{equation}
Moreover, assume that Lindeberg's condition is satisfied, that is, for all $\varepsilon>0$,
\begin{equation}\label{Lindecondi}
\sum_{k=1}^{n}
\mathbb{E}(
\|V_n \Delta M_k\|^{2}
\mathbb{I}_{\{\|V_n\Delta M_k\|>\varepsilon\}}| \mathcal{F}_{k-1})
\xrightarrow{p} 0,
\end{equation}
where $\Delta M_n=M_n-M_{n-1}$, $M_0=0$. Then, $V_n M_n
\xrightarrow{d}
\mathcal{N}(0,V)$.
\end{theorem}

A simplified version of Theorem 2.1 of Chaabane and Maaouia (2000) is as follows (see Bercu and Laulin (2021)):
\begin{theorem}\label{qsl}
Let $\{M_n, \mathcal{F}_n\}_{n\ge1}$ be a locally square-integrable martingale in
$\mathbb{R}^{d}$ with predictable quadratic variation $\{\langle M \rangle_n\}_{n\ge1}$.
Let $\{V_n\}_{n\ge1}$ be a sequence of non-random positive definite diagonal matrices of order $d$ such that its diagonal entries decrease to $0$ at polynomial rates. Assume that conditions \eqref{Touaticondi1} and \eqref{Lindecondi} hold a.s. Moreover, suppose that there exists $\eta\in(1,2]$ such that
\begin{equation}\label{qslcondi3}
\sum_{n=1}^{\infty}\frac{1}{(\log (\det V_n^{-1})^2)^\eta}\mathbb{E}(\|V_n\Delta M_n\|^{2\eta}|\mathcal{F}_{n-1})<\infty\ \ \text{a.s.}
\end{equation}
Then, we have the following quadratic strong law:
\begin{equation}\label{qslfin}
\lim_{n\to \infty}\frac{1}{\log (\det V_n^{-1})^2}\sum_{k=1}^{n}\Big(\frac{(\det V_k)^2-(\det V_{k+1})^2}{(\det V_k)^2}\Big)V_kM_kM_k^tV_k^t=V\ \ \text{a.s.}	
\end{equation}
\end{theorem}

\section{ERW with delayed amnesia}\label{interm}
In the standard elephant random walk, the memory mechanism remains unchanged throughout the evolution of the process with each past step being equally likely to be recalled at every time point. In many practical scenarios, the system initially possesses complete access to its history. However, after a certain stage the retrieval mechanism becomes selective or weakened, reducing the influence of remote past events. To model such situations, we introduce a modified elephant random walk, namely, the ERW with delayed amnesia, in which the memory mechanism undergoes a deterministic transition after a fixed time. That is, the walker follows the standard uniform memory rule upto a fixed time step and subsequently switches to an amnesic memory regime governed by a certain probability distribution. 

We define the ERW with delayed amnesia $\{S_n\}_{n\geq 1}$ by $S_n\coloneqq X_1+X_2+\dots+X_n$, $n\geq1$,
where the increment variables $\{X_n\}_{n\geq 1}$ have the following distribution:
For $n=1$, we have
\begin{equation*}
X_1=\begin{cases}
	1\ \ \text{with probability}\ q,\\
	-1\ \ \text{with probability}\ 1-q,
\end{cases}
\end{equation*}
where $0\le q\le 1$.

Let $a\in\mathbb{N}$ be fixed. For $1\leq n\leq a-1$, the process evolves according to the standard ERW mechanism. That is,
\begin{equation*}
X_{n+1}=\begin{cases}
		X_{U(n)}\ \ \text{with probability}\ p,\\
		-X_{U(n)}\ \ \text{with probability}\ 1-p,
	\end{cases}
\end{equation*}
where $0\le p\le1$ and $U(n)\sim \text{Uniform}\{1,2,\dots,n\}$. And, for $n\ge a$, it evolves as follows:
\begin{equation}\label{Xn+1}
X_{n+1}=\begin{cases}
		X_{\beta(n)}\ \ \text{with probability}\ p,\\
		-X_{\beta(n)}\ \ \text{with probability}\ 1-p,
	\end{cases}
\end{equation}
where 
\begin{equation}\label{betadef}
\mathbb{P}\{\beta(n)=k\}=\frac{a\Gamma(k)\Gamma(n-a+1)}{\Gamma(k-a+1)\Gamma(n+1)},\ \  k=a, a+1, \dots, n.
\end{equation}
Here, we assume that the randomly chosen steps in two memory mechanisms $U(n)$ and $\beta(n)$ are independent, and they are independent of $\{X_1,X_2,\dots$, $X_n\}$.
\begin{remark}
The parameter $a$ is the switching time between the two memory mechanisms. On substituting $a=1$ in \eqref{betadef}, we get the uniform distribution on $\{1,2,\dots,n\}$, that is, $\beta(n)\overset{d}{=}U(n)$. So, in this case, the ERW with delayed amnesia reduces to the standard ERW. 
\end{remark}

\begin{remark}
For fixed $k$, we have $\mathbb{P}\{\beta(n)=k\}\sim a\Gamma(k)/(\Gamma(k-a+1))n^{-a}$ as $n\to\infty$. So, the probability of recalling a fixed step decays polynomially as $a\Gamma(k)/(\Gamma(k-a+1))n^{-a}$.  Hence, increasing the amnesia parameter $a$ accelerates the loss of influence of remote memories. Also, $\Gamma(k)/ \Gamma(k-a+1)\sim k^{a-1}$ as $k\to\infty$. So, the memory kernel assigns increasingly larger weights to recent times.
\end{remark}

Let $\mathcal{F}_n=\sigma(X_1,X_2,\dots,X_n)$ be the $\sigma$-algebra generated by the history till $n$-th step. For $1\leq n\leq a-1$, from Eq. (7) of Sch\"utz and Trimper (2004), we have
\begin{equation*}
\mathbb{E}(X_{n+1}|\mathcal{F}_n)=\frac{2p-1}{n}\sum_{k=1}^{n}X_k.
\end{equation*}
And, for $n\ge a$, from \eqref{Xn+1}, we have the following for all $G\in\mathcal{F}_n$:
\begin{align}
\int_{G}\mathbb{E}(X_{n+1}|\mathcal{F}_n)\,\mathrm{d}\mathbb{P}
&=(2p-1)\mathbb{E}(X_{\beta(n)}\mathbb{I}_G)\nonumber\\
&=(2p-1)\sum_{k=a}^{n}\mathbb{E}(X_{\beta(n)}\mathbb{I}_G|\beta(n)=k)\mathbb{P}\{\beta(n)=k\}\nonumber\\
&=\frac{a(2p-1)\Gamma(n-a+1)}{\Gamma(n+1)}\int_{G}\sum_{k=a}^{n}\frac{\Gamma(k)}{\Gamma(k-a+1)}X_k\,\mathrm{d}\mathbb{P},\label{condiXn}
\end{align}
where the last step follows on using \eqref{betadef} and the independence of $\beta(n)$ with $\{X_1,X_2,\dots$, $X_n\}$. By applying the Radon-Nikodym theorem to \eqref{condiXn}, we get
\begin{equation}\label{XnYn}
\mathbb{E}(X_{n+1}|\mathcal{F}_n)=\frac{a(2p-1)\Gamma(n-a+1)}{\Gamma(n+1)}Y_n,
\end{equation}
where
\begin{equation}\label{Yn}
Y_n=\sum_{k=a}^{n}\frac{\Gamma(k)}{\Gamma(k-a+1)}X_k.
\end{equation}
Therefore,
\begin{equation}\label{condiYn}
\mathbb{E}(Y_{n+1}|\mathcal{F}_n)=\Big(1+\frac{a(2p-1)}{n-a+1}\Big)Y_n.
\end{equation}
Let $b_a=1$ and
\begin{equation}\label{b_n}
b_n=\prod_{k=a}^{n-1}\Big(1+\frac{a(2p-1)}{k-a+1}\Big)^{-1},
\end{equation}
for all $n>a$. Also, let us define
\begin{equation}\label{Mn}
M_n\coloneqq b_nY_n,\ \  n\ge a.
\end{equation}
Then, by using \eqref{condiYn}, it follows that $\{M_n,\mathcal{F}_n\}_{n\ge a}$ is a martingale. 

To study the limiting behaviour of ERW with delayed amnesia, we consider another sequence of random variables $\{N_n\}_{n\ge a}$ defined as follows:
\begin{equation}\label{Nn}
	N_n\coloneqq S_n+c_nY_n,
\end{equation}
where 
\begin{equation}\label{c_ndef}
	c_n=\frac{a(2p-1)\Gamma(n-a+1)}{(2a(1-p)-1)\Gamma(n)}
\end{equation} 
such that $2a(1-p)\ne 1$. By using \eqref{XnYn} and \eqref{condiYn}, we have
\begin{equation*}
	\mathbb{E}(N_{n+1}|\mathcal{F}_n)
	=S_n+c_nY_n=N_n.
\end{equation*}
Thus, $\{N_n,\mathcal{F}_n\}_{n\ge a}$ is a martingale. 

Now, let us consider the following martingale differences:
\begin{equation}\label{DeltaMn}
\Delta M_n=M_n-M_{n-1}, \ n\ge a,
\end{equation}
with $M_{a-1}=0$. By using \eqref{Yn} and \eqref{b_n}, we get
\begin{align}
\Delta M_{n+1}&=b_{n+1}\Big(Y_{n+1}-\frac{b_n}{b_{n+1}}Y_n\Big)\nonumber\\
&=b_{n+1}\Big(\frac{\Gamma(n+1)}{\Gamma(n-a+2)}X_{n+1}-\frac{a(2p-1)}{n-a+1}Y_n\Big)\nonumber\\
&=\frac{a(2p-1)}{2a(1-p)-1}b_{n+1}c_{n+1}^{-1}\varepsilon_{n+1},\label{DelMnEps}
\end{align}
where
\begin{equation}\label{epsin1}
\varepsilon_{n+1}=X_{n+1}-\frac{2a(1-p)-1}{n-a+1}c_{n+1}Y_n.
\end{equation}
Let $\varepsilon_a=X_a$. Then, from \eqref{DeltaMn} and \eqref{DelMnEps}, we have
\begin{equation}\label{Mn_additive}
M_n=\frac{a(2p-1)}{2a(1-p)-1}\sum_{k=a}^{n}b_kc_k^{-1}\varepsilon_k.
\end{equation}
For $n\ge a$, by using \eqref{XnYn} and \eqref{c_ndef} in \eqref{epsin1}, we get
\begin{equation}\label{epsPow1}
	\mathbb{E}(\varepsilon_{n+1}|\mathcal{F}_n)=0, 
\end{equation}	
and
\begin{equation}\label{epsilon2}
	\mathbb{E}(\varepsilon_{n+1}^2|\mathcal{F}_n)=1-\Big(\frac{2a(1-p)-1}{n-a+1}c_{n+1}Y_n\Big)^2.
\end{equation}	
So,
\begin{equation}\label{epsPow2}
	\sup_{n\ge a}\mathbb{E}(\varepsilon_{n+1}^2|\mathcal{F}_n)\le 1.
\end{equation}

The next result is obtained by using the following asymptotic relations:
\begin{equation}\label{asymb_n}
b_n\sim \Gamma(a(2p-1)+1)n^{-a(2p-1)} \ \ \text{as} \ n\to\infty
\end{equation}
and
\begin{equation}\label{asymc_n}
c_n\sim \frac{a(2p-1)}{2a(1-p)-1}n^{-a+1} \ \ \text{as} \ n\to\infty.
\end{equation}
\begin{proposition}\label{asymps_n}
For $n\ge a$, let
\begin{equation*}
s_n=\Big(\frac{a(2p-1)}{2a(1-p)-1}\Big)^2\sum_{k=a}^{n}(b_kc_k^{-1})^2.
\end{equation*}
Then,\vspace{.1cm}\\
\noindent{(i)} for $0\le p< (4a-1)/4a$, we have
\begin{equation*}
\lim_{n\to\infty}\frac{s_n}{n^{4a(1-p)-1}}=\frac{(\Gamma(a(2p-1)+1))^2}{4a(1-p)-1},
\end{equation*} 
\noindent{(ii)} for $p=(4a-1)/4a$, we have
\begin{equation*}
\lim_{n\to\infty}\frac{s_n}{\log n}=\Big(\Gamma\Big(a+\frac{1}{2}\Big)\Big)^2,\ \ \text{and}
\end{equation*}
\noindent{(iii)} for $(4a-1)/4a<p\le 1$, we have $\lim_{n\to\infty}s_n=l$,
where $l$ is some non zero constant.
\end{proposition}
The limiting behaviour of $\{M_n\}_{n\ge a}$ is determined by that of $\{s_n\}_{n\ge a}$. So, we consider three regimes based on the values of $p$ as given in Proposition \ref{asymps_n}. The cases $0\le p< (4a-1)/4a$, $p=(4a-1)/4a$ and $(4a-1)/4a<p\le 1$ correspond to the diffusive, critical and superdiffusive regimes, respectively.

\begin{remark}
Note that the critical value $p=1-1/4a$ approaches $1$ as $a\to \infty$. Thus,  increasing the amnesia parameter $a$ enlarges the range of $p$ for which the walk remains diffusive, as larger values of $p$ are required to reach the critical and superdiffusive regimes. This is consistent with the fact that the recall probability of a fixed past step decays as $a\Gamma(k)/(\Gamma(k-a+1))n^{-a}$, thereby weakening the influence of remote steps. However, for any positive integer $a$, the model undergoes a phase transition between the diffusive and superdiffusive regimes.
\end{remark}

Let $\Delta N_a=(a-1)X_a/(2a(1-p)-1)$ and $\Delta N_{n+1}\coloneqq N_{n+1}-N_n$, for all $n\ge a$. Then, from \eqref{Nn}, we have
\begin{equation}\label{DelNn}
	\Delta N_n=\frac{a-1}{2a(1-p)-1}\varepsilon_n.
\end{equation}
Also, let $N_{a-1}=S_{a-1}$. Then, 
\begin{equation}
	N_n=S_{a-1}+\frac{a-1}{2a(1-p)-1}\sum_{k=a}^{n}\varepsilon_k.
\end{equation}

The next result follows on using \eqref{epsPow1} and \eqref{epsPow2} in Theorem \ref{thm_LLN2}. 
\begin{lemma}\label{Nn_O}
Let $\gamma> 0$ and $\{N_n\}_{n\ge a}$ be as defined in \eqref{Nn}. Then,
\begin{equation*}
(N_n-S_{a-1})^2=O(n(\log n)^{1+\gamma})\ \ \text{a.s.}
\end{equation*}
\end{lemma}

Let us consider the following predictable quadratic variations:
\begin{equation}\label{predMn}
\langle M\rangle_n=\sum_{k=a}^{n}\mathbb{E}((\Delta M_k)^2|\mathcal{F}_{k-1})
\end{equation} 
and
\begin{equation}\label{predNn}
\langle N\rangle_n=\sum_{k=a}^{n}\mathbb{E}((\Delta N_k)^2|\mathcal{F}_{k-1}),
\end{equation}
corresponding to $\{M_n\}_{n\ge a}$ and $\{N_n\}_{n\ge a}$, respectively, such that $\langle M\rangle_{a-1}=0$ and $\langle N\rangle_{a-1}=0$.

The following limiting results hold for the predictable quadratic variation of $\{M_n\}_{n\ge a}$:
\begin{lemma}
For $0\le p<(4a-1)/4a$, we have
\begin{equation}\label{angMndiff}
\lim_{n\to \infty}\frac{\langle M\rangle_n}{n^{4a(1-p)-1}}=\frac{(\Gamma(a(2p-1)+1))^2}{4a(1-p)-1}\ \ \text{a.s.}
\end{equation}
and for $p=(4a-1)/4a$, we have
\begin{equation}\label{angMncrit}
\lim_{n\to \infty}\frac{\langle M\rangle_n}{\log n}=\Big(\Gamma\Big(a+\frac{1}{2}\Big)\Big)^2\ \ \text{a.s.}
\end{equation}
\end{lemma}
\begin{proof}
From \eqref{DelMnEps} and \eqref{predMn}, we have 
\begin{align}
\langle M\rangle_n&=\sum_{k=a}^{n}\Big(\frac{a(2p-1)}{2a(1-p)-1}\Big)^2(b_kc_k^{-1})^2\mathbb{E}(\varepsilon_k^2|\mathcal{F}_{k-1})\nonumber\\
&=\Big(\frac{a(2p-1)}{2a(1-p)-1}\Big)^2\sum_{k=a}^{n}(b_kc_k^{-1})^2-(a(2p-1))^2\sum_{k=a}^{n-1}\Big(\frac{b_{k+1}Y_{k}}{k-a+1}\Big)^2,\label{mathangMn}
\end{align}
where the last step follows from \eqref{epsilon2}.

For $0\le p<(4a-1)/4a$,
from Proposition \ref{asymps_n}(i), we have
\begin{equation}\label{sumbkck}
\sum_{k=a}^{n}(b_kc_k^{-1})^2\sim \Big(\frac{(2a(1-p)-1)\Gamma(a(2p-1)+1)}{a(2p-1)}\Big)^2 \frac{n^{4a(1-p)-1}}{4a(1-p)-1}.
\end{equation}
By using \eqref{Mn_additive}, \eqref{epsPow1}, \eqref{epsPow2} and \eqref{sumbkck} in Theorem \ref{thm_LLN2}, we get the following for all $\gamma>0$: 
\begin{equation}\label{Mn_O}
M_n^2=O(n^{4a(1-p)-1}(\log n)^{1+\gamma})\ \ \text{a.s.}
\end{equation}
Now, from \eqref{asymb_n} and \eqref{Mn_O}, we get $Y_n^2=O(n^{2a-1}(\log n)^{1+\gamma})$ a.s. So, 
\begin{equation}\label{bkYk}
\Big(\frac{b_{k+1}Y_{k}}{k-a+1}\Big)^2=O(k^{4a(1-p)-3}(\log k)^{1+\gamma})\ \ \text{a.s.}
\end{equation}
Thus, by using \eqref{sumbkck} and \eqref{bkYk} in \eqref{mathangMn}, we get \eqref{angMndiff}.

Similarly, for $p=(4a-1)/4a$, by using \eqref{Mn_additive}, \eqref{epsPow1} and \eqref{epsPow2} in Theorem \ref{thm_LLN2}, we get
\begin{equation}\label{Mn_Ocri}
	M_n^2=O((\log\log n)^{1+\gamma}\log n)\ \ \text{a.s.,}
\end{equation} 
for all $\gamma>0$. Now, on using \eqref{asymb_n} in \eqref{Mn_Ocri}, we get
\begin{equation*}
	Y_n^2=O(n^{2a-1}(\log\log n)^{1+\gamma}\log n)\ \ \text{a.s.}
\end{equation*}
Thus,
\begin{equation*}
	\Big(\frac{b_{k+1}Y_k}{k-a+1}\Big)^2=O\Big(\frac{(\log\log k)^{1+\gamma}\log k}{k^2}\Big)\ \ \text{a.s.}
\end{equation*}
So,
\begin{equation}\label{bkyk1}
	\sum_{k=a}^{n-1}\Big(\frac{b_{k+1}Y_k}{k-a+1}\Big)^2=O(1)\ \ \text{a.s.}
\end{equation}
Also, from Proposition \ref{asymps_n}(ii), we have
\begin{equation}\label{sumbkckcrit}
	\sum_{k=a}^{n}(b_kc_k^{-1})^2\sim \Big(\frac{\Gamma(a+1/2)}{2a-1}\Big)^2 \log n.
\end{equation}
Finally, by using \eqref{bkyk1} and \eqref{sumbkckcrit} in \eqref{mathangMn}, we get \eqref{angMncrit}. This completes the proof.
\end{proof}

The following limiting results hold for the predictable quadratic variation of $\{N_n\}_{n\ge a}$:
\begin{lemma}
For $0\le p<(4a-1)/4a$, we have
\begin{equation}\label{angNndiff}
	\lim_{n\to \infty}\frac{\langle N\rangle_n}{n}=\Big(\frac{a-1}{2a(1-p)-1}\Big)^2\ \ \text{a.s.}
\end{equation}
and for $p=(4a-1)/4a$, we have
\begin{equation}\label{angNncrit}
	\lim_{n\to \infty}\frac{\langle N\rangle_n}{n}=4(a-1)^2\ \ \text{a.s.}
\end{equation}
\end{lemma}
\begin{proof}
By using \eqref{epsilon2} and \eqref{DelNn} in \eqref{predNn}, we get
\begin{equation}\label{Nnangle}
\langle N\rangle_n=\Big(\frac{	a-1}{2a(1-p)-1}\Big)^2(n-a+1)-(a-1)^2\sum_{k=a}^{n-1}\Big(\frac{c_{k+1}Y_{k}}{k-a+1}\Big)^2.
\end{equation}
For $0\le p<(4a-1)/4a$, by using \eqref{asymb_n} and \eqref{asymc_n} in \eqref{Mn_O}, we obtain
\begin{equation*}
	(n^{-1}c_nY_n)^2=O(n^{-1}(\log n)^{1+\gamma})\ \ \text{a.s.}
\end{equation*}
So,
\begin{equation}\label{c_ndiff}
\sum_{k=a}^{n-1}\Big(\frac{c_{k+1}Y_{k}}{k-a+1}\Big)^2=O((\log n)^{2+\gamma})\ \ \text{a.s.}
\end{equation}
Thus, \eqref{angNndiff} is obtained by using \eqref{c_ndiff} in \eqref{Nnangle}.

Similarly, for $p=(4a-1)/4a$, by using \eqref{asymb_n} and \eqref{asymc_n} in \eqref{Mn_Ocri}, we get
\begin{equation}\label{cnYnsq}	\Big(\frac{c_nY_n}{\sqrt{n}\log n}\Big)^2=O\Big(\frac{(\log\log n)^{1+\gamma}}{\log n}\Big)\ \ \text{a.s.}
\end{equation}
Therefore,
\begin{equation}\label{c_ncrit}
\sum_{k=a}^{n-1}\Big(\frac{c_{k+1}Y_{k}}{k-a+1}\Big)^2=O((\log n)^2(\log \log n)^{1+\gamma})\ \ \text{a.s.}
\end{equation}
Thus, \eqref{angNncrit} is obtained by using \eqref{c_ncrit} in \eqref{Nnangle}. This completes the proof.
\end{proof}

Let us consider the following vector martingale:
\begin{equation}\label{twotemmar}
\mathcal{M}_n=\begin{pmatrix}
		M_n\vspace{.2cm}\\
		N_n
	\end{pmatrix},
\end{equation}
where $\{M_n\}_{n\ge a}$ and $\{N_n\}_{n\ge a}$ are locally square-integrable martingales given in \eqref{Mn} and \eqref{Nn}, respectively. For $n\ge a$, from \eqref{DelMnEps} and \eqref{DelNn}, we have 
{\footnotesize\begin{equation*}
(\Delta\mathcal{M}_{n+1})(\Delta\mathcal{M}_{n+1})^t=\Big(\frac{1}{2a(1-p)-1}\Big)^2\begin{pmatrix}
	(a(2p-1)b_{n+1}c_{n+1}^{-1})^2 & a(2p-1)(a-1)b_{n+1}c_{n+1}^{-1}\vspace{.2cm}\\
	a(2p-1)(a-1)b_{n+1}c_{n+1}^{-1} & (a-1)^2
\end{pmatrix}\varepsilon_{n+1}^2.
\end{equation*}}
So, the predictable quadratic variation of $\{\mathcal{M}_n\}_{n\ge a}$ is 
\begin{align}
	\left\langle\mathcal{M}\right\rangle_{n}&=\Big(\frac{1}{2a(1-p)-1}\Big)^2\sum_{k=a}^{n}Q_k\mathbb{E}(\varepsilon_{k}^2|\mathcal{F}_{k-1})\nonumber\\
	&=\Big(\frac{1}{2a(1-p)-1}\Big)^2\Big(\sum_{k=a}^{n}Q_k-\sum_{k=a}^{n-1}Q_{k+1}\Big(\frac{c_{k+1}Y_k}{k-a+1}\Big)^2\Big),\label{predquad}
\end{align}
where last equality follows from \eqref{epsilon2}. Here, $\left\langle\mathcal{M}\right\rangle_{a-1}=0$ and
\begin{equation}\label{qk}
Q_k=\begin{pmatrix}
	(a(2p-1)b_{k}c_{k}^{-1})^2 & a(2p-1)(a-1)b_{k}c_{k}^{-1}\vspace{.2cm}\\
	a(2p-1)(a-1)b_{k}c_{k}^{-1} & (a-1)^2
	\end{pmatrix}.
\end{equation}
Note that $\langle M\rangle_n$ and $\langle N\rangle_n$ are the $(1,1)$-th and $(2,2)$-th entries of $\left\langle\mathcal{M}\right\rangle_{n}$ given in \eqref{predquad}, respectively.

\subsection{Limiting results related to the walk}
Here, we obtain the asymptotic results for ERW with delayed amnesia in three different regimes. Recall that the diffusive regime corresponds to $0\le p< (4a-1)/4a$, critical regime corresponds to $ p= (4a-1)/4a$ and supperdiffusive regime to $ (4a-1)/4a<p\le1$.
\subsubsection{The diffusive regime} First, we derive the limiting results for the walk in diffusive regime.
\begin{theorem}\label{thmdiffu}
Let $0\le p< (4a-1)/4a$. Then, $\lim_{n\to\infty}S_n/n=0$ a.s.
\end{theorem}
\begin{proof}
By using \eqref{asymb_n} and \eqref{asymc_n} in \eqref{Mn_O}, we get
\begin{equation*}
(n^{-1}c_nY_n)^2=O(n^{-1}(\log n)^{1+\gamma})\ \ \text{a.s.},
\end{equation*}
for all $\gamma>0$. 
So, it follows that
\begin{equation}\label{cnYn}
\lim_{n\to\infty}\frac{c_nY_n}{n}=0 \ \ \text{a.s.}
\end{equation}
By using \eqref{Nn} in  Lemma \ref{Nn_O}, we have
\begin{equation}
(n^{-1}(S_n+c_nY_n-S_{a-1}))^2=O(n^{-1}(\log n)^{1+\gamma}) \ \ \text{a.s.}
\end{equation}
Thus,
\begin{equation}\label{SncnYn}
\lim_{n\to\infty}\frac{1}{n}(S_n+c_nY_n-S_{a-1})=0 \ \ \text{a.s.}
\end{equation}
Finally, by using \eqref{cnYn} in \eqref{SncnYn}, we obtain the required result.
\end{proof}
Now, we obtain the law of iterated logarithm for $\{S_n\}_{n\ge1}$.
\begin{theorem}\label{lildiff}
Let $0\le p<(4a-1)/4a$. Then,
\begin{equation*}
\limsup_{n\to\infty}\frac{S_n}{\sqrt{2n^{2a-1}\log\log n}}=\begin{cases}
	0\ \text{for}\ a>1,\ \text{a.s.},\\
	\frac{1}{\sqrt{3-4p}}\ \text{for}\ a=1,\ \text{a.s.}
\end{cases}
\end{equation*}
\end{theorem}
\begin{proof}
Let $\widetilde{M}_n=M_n-\mathbb{E}(M_a)$, $n\ge a$. Then,  $\{\widetilde{M}_n, \mathcal{F}_n\}_{n\ge a}$ is a zero mean martingale.

From \eqref{DelMnEps}, we have
\begin{equation}\label{deltilM_k}
	\Delta \widetilde{M}_k=\widetilde{M}_k-\widetilde{M}_{k-1}=\frac{a(2p-1)}{2a(1-p)-1}b_{k}c_{k}^{-1}\varepsilon_{k}.
\end{equation}
Also, let $t_n^2=\sum_{k=a}^{n}\mathbb{E}((\Delta\widetilde{M}_k)^2|\mathcal{F}_{k-1})$. Then,
\begin{equation*}
	t_n^2=\Big(\frac{a(2p-1)}{2a(1-p)-1}\Big)^2\sum_{k=a}^{n}(b_{k}c_{k}^{-1})^2-(a(2p-1))^2\sum_{k=a}^{n-1}\Big(\frac{b_{k+1}Y_k}{k-a+1}\Big)^2,
\end{equation*}
where we have used \eqref{mathangMn}. From \eqref{angMndiff} and \eqref{epsin1}, we have $\lim_{n\to\infty}t_n^2=\infty$ a.s. and $\sup_{a\le k\le n}|\varepsilon_{k}|\le c$ for some constant $c$, respectively. So,
	\begin{equation*}
	|\Delta \widetilde{M}_n|\le\Big|\frac{ac(2p-1)b_{n}c_{n}^{-1}}{2a(1-p)-1}\Big|,
\end{equation*}
where we have used \eqref{deltilM_k}.

Let $u_n=\sqrt{2\log\log t_n^2}$ and
\begin{equation*}
	K_n=\Big|\frac{ac(2p-1)b_{n}c_{n}^{-1}}{(2a(1-p)-1)t_n}\Big|u_n.
\end{equation*}
Then,
\begin{equation*}
	|\Delta \widetilde{M}_n|\le\frac{K_n|t_n|}{u_n}.
\end{equation*}
As $t_n^2$ is $\mathcal{F}_{n-1}$-measurable, so is $K_n$.
From \eqref{asymb_n}, \eqref{asymc_n} and \eqref{angMndiff}, we get 
\begin{equation*}
	K_n=O\Big(\Big(\frac{2\log\log t_n^2}{n}\Big)^{1/2}\Big)=O\Big(\Big(\frac{2\log\log n}{n}\Big)^{1/2}\Big)\ \ \text{a.s.}
\end{equation*}
So, $\lim_{n\to\infty}K_n=0$ a.s.
Thus, from Theorem \ref{stoutthm2}, we have
\begin{equation*}
	\limsup_{n\to\infty}\frac{\widetilde{M}_n}{n^{2a(1-p)-1/2}\sqrt{2\log\log n}}=\frac{\Gamma(a(2p-1)+1)}{\sqrt{4a(1-p)-1}}\ \ \text{a.s.}
\end{equation*}
That is,
\begin{equation}\label{limsupmn}
	\limsup_{n\to\infty}\frac{M_n}{\sqrt{2n^{4a(1-p)-1}\log\log n}}=\frac{\Gamma(a(2p-1)+1)}{\sqrt{4a(1-p)-1}}\ \ \text{a.s.}
\end{equation}
By using \eqref{asymb_n} in  \eqref{limsupmn}, we get
\begin{equation}\label{limsupYn}
	\limsup_{n\to\infty}\frac{Y_n}{\sqrt{2n^{2a-1}\log\log n}}=\frac{1}{\sqrt{4a(1-p)-1}}\ \ \text{a.s.}
\end{equation}
Similarly, we obtain
\begin{equation*}
	\limsup_{n\to\infty}\frac{N_n}{\sqrt{2n\log\log n}}=\frac{a-1}{|2a(1-p)-1|}\ \ \text{a.s.}
\end{equation*}
Therefore,
\begin{equation}\label{limsupNn}
	\limsup_{n\to\infty}\frac{N_n}{\sqrt{2n^{2a-1}\log\log n}}=0\ \ \text{a.s.}
\end{equation}
Finally, by using  \eqref{Nn}, \eqref{asymc_n} and \eqref{limsupYn} in  \eqref{limsupNn}, we get the required result.
\end{proof}
\begin{remark}
	For $a=1$, Theorem \ref{lildiff} reduces to Eq. (3.3) of Bercu (2018).
\end{remark}

Next, we obtain the asymptotic normality of $\{S_n\}_{n\ge1}$.
\begin{theorem}\label{thmasynrdiff}
Let $0\le p<(4a-1)/4a$. Then,
\begin{equation}\label{Snnordiff}
\frac{S_n}{\sqrt{n}}\xrightarrow{d}\mathcal{N}\Big(0,\frac{K}{(2a(1-p)-1)^2}\Big),
\end{equation}
where 
\begin{equation}\label{K}
K=	\frac{(a(2p-1))^2}{4a(1-p)-1}+\frac{(1-2p)(a-1)}{1-p}+(a-1)^2.
\end{equation}
\end{theorem}
\begin{proof}
Let $\{A_n\}_{n\ge a}$ be a sequence of positive definite diagonal matrices given by
\begin{equation*}
A_n=\frac{1}{\sqrt{n}}\begin{pmatrix}
	b_n^{-1}c_n & 0\vspace{.2cm}\\
	0 & 1
\end{pmatrix}.
\end{equation*}
Then, from \eqref{qk}, we get
{\footnotesize\begin{equation*}
A_n\Big(\sum_{k=a}^{n}Q_k\Big)A_n^t=\frac{1}{n}\begin{pmatrix}
(a(2p-1))^2	b_n^{-1}c_n\sum_{k=a}^{n}(b_{k}c_{k}^{-1})^2 & a(2p-1)(a-1)b_n^{-1}c_n\sum_{k=a}^{n}b_{k}c_{k}^{-1}\vspace{.2cm}\\
	a(2p-1)(a-1)b_n^{-1}c_n\sum_{k=a}^{n}b_{k}c_{k}^{-1} & (a-1)^2(n-a+1)
\end{pmatrix}.
\end{equation*}}
From \eqref{asymb_n} and \eqref{asymc_n}, we have
\begin{equation*}
b_n^{-1}c_n\sum_{k=a}^{n}b_{k}c_{k}^{-1}\sim \frac{n}{2a(1-p)} 
\end{equation*}
and 
\begin{equation*}
A_n^tA_n\sim\begin{pmatrix}
	\frac{n^{4a(p-1)+1}}{((2a(1-p)-1)\Gamma(a(2p-1)))^2} & 0\vspace{.2cm}\\
	0 & 1/n
\end{pmatrix}.
\end{equation*}
Note that $\lambda_{\max}(A_n^tA_n) \to 0$ as $n\to \infty$. So, $\|A_n\|\to 0$ as $n\to \infty$.
Also,
{\small\begin{align*}
A_n&\Big(\sum_{k=a}^{n-1}Q_{k+1}\Big(\frac{c_{k+1}Y_k}{k-a+1}\Big)^2\Big)A_n^t\\
&=\ \ \begin{pmatrix}
	\frac{(a(2p-1)b_{n}^{-1}c_{n})^2}{n}\sum_{k=a}^{n-1}\Big(\frac{b_{k+1}Y_k}{k-a+1}\Big)^2 & \frac{a(2p-1)(a-1)b_{n}^{-1}c_{n}}{n}\sum_{k=a}^{n-1}b_{k+1}c_{k+1}\Big(\frac{Y_k}{k-a+1}\Big)^2\vspace{.2cm}\\
	\frac{a(2p-1)(a-1)b_{n}^{-1}c_{n}}{n}\sum_{k=a}^{n-1}b_{k+1}c_{k+1}\Big(\frac{Y_k}{k-a+1}\Big)^2 & \frac{(a-1)^2}{n}\sum_{k=a}^{n-1}\Big(\frac{c_{k+1}Y_k}{k-a+1}\Big)^2
\end{pmatrix}
\end{align*}}
is a matrix whose each entry goes to zero as $n\to\infty$. Thus,
\begin{equation}\label{anmnt}
\lim_{n\to\infty}A_n\left\langle\mathcal{M}\right\rangle_{n}A_n^t=A\ \ \text{a.s.},
\end{equation}
which is obtained by using \eqref{asymb_n} and \eqref{asymc_n} in \eqref{predquad}. Here, $A$ is a symmetric and positive semi-definite matrix given by
\begin{equation}\label{A_diffu}
A=\frac{1}{(2a(1-p)-1)^2}\begin{pmatrix}
	\frac{(a(2p-1))^2}{4a(1-p)-1} & \frac{(2p-1)(a-1)}{2(1-p)}\vspace{.2cm}\\
	\frac{(2p-1)(a-1)}{2(1-p)} & (a-1)^2
\end{pmatrix}.
\end{equation}
Also, for $n\ge a$, we have 
\begin{equation*}
\|A_n\Delta\mathcal{M}_k\|^2=\frac{1}{n}((a(2p-1) b_n^{-1}c_nb_kc_k^{-1}\varepsilon_k)^2+((a-1)\varepsilon_k)^2).
\end{equation*}
So, for all $\varepsilon>0$,
\begin{equation}\label{ineqlind}
\sum_{k=a}^{n}
\mathbb{E}(
\|A_n \Delta \mathcal{M}_k\|^{2}
\mathbb{I}_{\{\|A_n\Delta \mathcal{M}_k\|>\varepsilon\}}| \mathcal{F}_{k-1})\le \frac{1}{\varepsilon^2}\mathbb{E}(
\|A_n \Delta \mathcal{M}_k\|^{4}| \mathcal{F}_{k-1}).
\end{equation} 
From \eqref{epsin1}, we have $\sup_{a\le k\le n}|\varepsilon_k|\le c$, where $c$ is some constant. Therefore,
{\begin{align}\label{norm4}
\|A_n \Delta \mathcal{M}_k\|^{4}&\le \frac{c^4}{n^2}((a(2p-1)b_n^{-1}c_nb_kc_k^{-1})^4\nonumber\\
&\ \ +2(a(2p-1)(a-1)b_n^{-1}c_nb_kc_k^{-1})^2+(a-1)^4).
\end{align}}
By using \eqref{asymb_n}, \eqref{asymc_n} and \eqref{norm4} in \eqref{ineqlind}, we get
\begin{equation}\label{anmnt1}
\lim_{n\to \infty}\sum_{k=a}^{n}
\mathbb{E}(
\|A_n \Delta \mathcal{M}_k\|^{2}
\mathbb{I}_{\{\|A_n\Delta \mathcal{M}_k\|>\varepsilon\}}| \mathcal{F}_{k-1})= 0\ \ \text{a.s.}
\end{equation}
Thus, by using Theorem \ref{thm_asymnorm}, we have
\begin{equation*}
\frac{1}{\sqrt{n}}\begin{pmatrix}
	c_nY_n\vspace{.2cm}\\
	N_n
\end{pmatrix}	\xrightarrow{d} \mathcal{N}(0,A).
\end{equation*}
Let $v=\begin{pmatrix}
	-1 &	1
\end{pmatrix}^t$. Then,
\begin{equation}
\frac{v^t}{\sqrt{n}}\begin{pmatrix}
	c_nY_n\vspace{.2cm}\\
	N_n
\end{pmatrix}	\xrightarrow{d} \mathcal{N}\Big(0,K\Big).
\end{equation}
This completes the proof.
\end{proof}
\begin{remark}
On substituting $a=1$ in \eqref{Snnordiff}, we get $S_n/\sqrt{n}\xrightarrow{d}\mathcal{N}(0,1/(3-4p))$, which coincides with Eq. (3.5) of Bercu (2018). 
\end{remark}

In the next result, we obtain the quadratic strong law for $\{S_n\}_{n\ge 1}$.
\begin{theorem}
Let $0\le p<(4a-1)/4a$. Then,
{\small\begin{equation*}
\lim_{n\to \infty}\frac{1}{\log n}\sum_{k=1}^{n}\frac{S_k^2}{k^2}=\frac{1}{(2a(1-p)-1)^2}\Big(\frac{(a(2p-1))^2}{4a(1-p)-1}+\frac{(1-2p)(a-1)}{1-p}+(a-1)^2\Big)\ \ \text{a.s.}
\end{equation*}}
\end{theorem}
\begin{proof}
We have
\begin{equation*}
A_n=\frac{1}{\sqrt{n}}\begin{pmatrix}
b_n^{-1}c_n & 0\vspace{.2cm}\\
	0 & 1
\end{pmatrix}.
\end{equation*}
So, $\det A_n^{-1}=b_nc_n^{-1}n$. Note that conditions given in \eqref{Touaticondi1} and \eqref{Lindecondi} are satisfied in \eqref{anmnt} and \eqref{anmnt1}, respectively. To check condition \eqref{qslcondi3}, we have
\begin{equation}\label{logdet}
\lim_{n\to \infty}\frac{\log (\det A_n^{-1})^2}{\log n}=4a(1-p),
\end{equation}
which follows from \eqref{asymb_n} and \eqref{asymc_n}. By using \eqref{asymb_n} and \eqref{asymc_n} in \eqref{norm4}, we get
\begin{equation*}
\|A_n\Delta\mathcal{M}_n\|^4=O(n^{-2}).
\end{equation*}
Thus,
\begin{equation}
\sum_{n=a+1}^{\infty}\frac{1}{(\log n)^2}\mathbb{E}(\|A_n\Delta\mathcal{M}_n\|^4|\mathcal{F}_{n-1})=O\Big(\sum_{n=2}^{\infty}\frac{1}{(n\log n)^2}\Big)\ \ \text{a.s.}
\end{equation}
Consequently, for $\eta=2$,  the condition given in \eqref{qslcondi3} is satisfied. Also, we have 
\begin{equation}\label{detfrac}
\lim_{n\to \infty} n\Big(\frac{(\det A_n)^2-(\det A_{n+1})^2}{(\det A_n)^2}\Big)=4a(1-p),
\end{equation}
which follows from \eqref{b_n} and \eqref{c_ndef}. Since all the conditions of Theorem \ref{qsl} are satisfied, by using \eqref{logdet} and  \eqref{detfrac} in \eqref{qslfin}, we obtain
\begin{equation}\label{qslpenul}
\lim_{n\to \infty}	\frac{1}{\log n}\sum_{k=a}^{n}A_k\mathcal{M}_k\mathcal{M}_k^tA_k^t=A\ \ \text{a.s.},
\end{equation}
where $A$ is as given in \eqref{A_diffu}.

Let $v=\begin{pmatrix}
	-1 &	1
\end{pmatrix}^t$.
Then, from \eqref{qslpenul}, we get
\begin{equation*}
	\lim_{n\to \infty}\frac{1}{\log n}\sum_{k=a}^{n}\frac{S_k^2}{k^2}=v^tAv\ \ \text{a.s.}
\end{equation*}
Equivalently,
\begin{equation*}
	\lim_{n\to \infty}\frac{1}{\log n}\sum_{k=1}^{n}\frac{S_k^2}{k^2}=v^tAv\ \ \text{a.s.},
\end{equation*}
which reduces to the required result.
\end{proof}

From \eqref{XnYn} and \eqref{condiYn}, we get the following asymptotic result:

\begin{lemma} \label{Yn2asym}
For $n\ge a$, let $Y_n$ be as defined in \eqref{Yn}. Then,
\begin{equation*}
\mathbb{E}(Y_n^2)\sim\begin{cases}\frac{n^{2a-1}}{4a(1-p)-1}\ \ \text{for}\ 0\le p<(4a-1)/4a,\vspace{.1cm}\\
	n^{2a-1}\log n\ \ \text{for}\  p=(4a-1)/4a,\vspace{.1cm}\\
	c'n^{2a(2p-1)}\ \ \text{for}\  (4a-1)/4a<p\le1,
\end{cases}
\end{equation*}	
where $c'$ is some non zero constant.	
\end{lemma}	
Next, we obtain an asymptotic result for the mean square displacement of $\{S_n\}_{n\ge 1}$.
\begin{theorem}\label{meansqdiff}
Let $0\le p<(4a-1)/4a$. Then,
\begin{equation*}
\mathbb{E}(S_n^2)\sim \frac{n}{(2a(1-p)-1)^2}\Big(\frac{(a(2p-1))^2}{4a(1-p)-1}+\frac{(1-2p)(a-1)}{1-p}+(a-1)^2\Big),\ \ \text{as} \ n\to\infty.
\end{equation*}
\end{theorem}
\begin{proof}
Let $v=\begin{pmatrix}
	-1 &	1
\end{pmatrix}^t$ and
\begin{equation*}
	A_n=\frac{1}{\sqrt{n}}\begin{pmatrix}
		b_n^{-1}c_n & 0\vspace{.2cm}\\
		0 & 1
	\end{pmatrix}.
\end{equation*}
Then, 
\begin{align*}
\frac{S_n^2}{n}&=v^tA_n\mathcal{M}_n(v^tA_n\mathcal{M}_n)^t\\
&=v^tA_n\begin{pmatrix}
M_n^2 & M_nN_n\\
M_nN_n & N_n^2
\end{pmatrix}A_n^tv\\
&=\frac{1}{n}((b_n^{-1}c_n)^2M_n^2-2b_n^{-1}c_nM_nN_n+N_n^2),
\end{align*}
which follows from \eqref{twotemmar}. So,
\begin{equation}\label{meandisdiff}
\mathbb{E}\Big(\frac{S_n^2}{n}\Big)=\frac{1}{n}((b_n^{-1}c_n)^2\mathbb{E}(M_n^2)-2b_n^{-1}c_n\mathbb{E}(M_nN_n)+\mathbb{E}(N_n^2)).
\end{equation}
Let
\begin{equation*}
\langle M,N\rangle_n=\sum_{k=a}^{n}\mathbb{E}(\Delta M_k\Delta N_k|\mathcal{F}_{k-1}),
\end{equation*}
which is $(1,2)$-th entry of \eqref{predquad}. By using \eqref{asymb_n}, \eqref{asymc_n} and Lemma \ref{Yn2asym} in \eqref{predquad}, we get
\begin{align}\label{threasym}
\left.\begin{aligned}
	\mathbb{E}(\langle M\rangle_{n})&\sim \frac{(\Gamma(a(2p-1)+1))^2}{4a(1-p)-1}n^{4a(1-p)-1},\\
	\mathbb{E}(\langle M,N\rangle_{n})&\sim \frac{(a-1)\Gamma(a(2p-1)+1)}{2a(1-p)(2a(1-p)-1)}n^{2a(1-p)},\\
	\mathbb{E}(\langle N\rangle_{n})&\sim \Big(\frac{a-1}{2a(1-p)-1}\Big)^2n.
\end{aligned}\right\}
\end{align}
Note that $M_n^2-\langle M\rangle_n$, $M_nN_n-\langle M,N\rangle_n$ and $N_n^2-\langle N\rangle_n$  are  martingales with respect to $\{\mathcal{F}_n\}_{n\ge a}$.
So, using \eqref{threasym}, we get
\begin{align}\label{threasym1}
\left.\begin{aligned}
	\mathbb{E}(M_n^2)&\sim\mathbb{E}(\langle M\rangle_n),\\
	\mathbb{E}(M_nN_n)&\sim\mathbb{E}(\langle M,N\rangle_n),\\
	\mathbb{E}(N_n^2)&\sim\mathbb{E}(\langle M\rangle_n).
\end{aligned}\right\}
\end{align}
Finally, the required result follows on using \eqref{threasym} and \eqref{threasym1} in \eqref{meandisdiff}.
\end{proof}

\subsubsection{The critical regime} Here, we discuss the results obtained for the ERW with delayed amnesia in critical regime.
\begin{theorem}\label{thmcirit}
Let $p=(4a-1)/4a$. Then,
\begin{equation*}
\lim_{n\to\infty}\frac{S_n}{\sqrt{n}\log n}=0\ \ \text{a.s.}
\end{equation*}
\end{theorem}
\begin{proof}
From \eqref{cnYnsq}, we get
\begin{equation}\label{Ynlogn}
\lim_{n\to\infty}\frac{c_nY_n}{\sqrt{n}\log n}=0\ \ \text{a.s.}
\end{equation}
By using \eqref{Nn} in Lemma \ref{Nn_O} for $\gamma=1/2$, we have
\begin{equation}\label{Snlogn}
\Big(\frac{S_n+c_nY_n-S_{a-1}}{\sqrt{n}\log n}\Big)^2=O\Big(\frac{1}{\sqrt{\log n}}\Big)\ \ \text{a.s.}
\end{equation}
Finally, the required result follows from \eqref{Ynlogn} and \eqref{Snlogn}. 
\end{proof}
\begin{theorem}\label{critlil}
	Let $p=(4a-1)/4a$. Then,
	\begin{equation}\label{Snloglog}
		\limsup_{n\to\infty}\frac{S_n}{\sqrt{2n\log n\log\log\log n}}=2a-1\ \ \text{a.s.}
	\end{equation}
\end{theorem}
\begin{proof}
Recall that for $n\ge a$, we have $\widetilde{M}_n=M_n-\mathbb{E}(M_a)$ and $t_n^2=\sum_{k=a}^{n}\mathbb{E}((\Delta\widetilde{M}_k)^2|\mathcal{F}_{k-1})$. From \eqref{angMncrit}, we have $\lim_{n\to\infty}t_n^2=\infty$ a.s. and as $u_n=\sqrt{2\log\log t_n^2}$, it follows that
\begin{equation*}
	K_n=\Big|\frac{ac(2p-1)b_{n}c_{n}^{-1}}{(2a(1-p)-1)t_n}\Big|u_n=O\Big(\Big({\frac{2\log\log\log n}{n\log n}}\Big)^{1/2}\Big)\ \ \text{a.s.}
\end{equation*}
So, $\lim_{n\to\infty}K_n=0$ a.s. Therefore, from  Theorem \ref{stoutthm2}, we get
\begin{equation*}
	\limsup_{n\to\infty}\frac{M_n}{\sqrt{2\log n\log\log\log n}}=\Gamma\Big(a+\frac{1}{2}\Big)\ \ \text{a.s.}
\end{equation*}
Thus,
\begin{equation}\label{M_nsup}
	\limsup_{n\to\infty}\frac{c_nY_n}{2n\log n\log\log\log n}=2a-1\ \ \text{a.s.},
\end{equation}
which follows from \eqref{asymb_n} and \eqref{asymc_n}. 
Analogously, we obtain
\begin{equation*}
	\limsup_{n\to\infty}\frac{N_n}{\sqrt{2n\log\log n}}=\frac{a-1}{|2a(1-p)-1|}\ \ \text{a.s.}
\end{equation*}
Therefore, 
\begin{equation}\label{N_nas}
	\lim_{n\to\infty}\frac{N_n}{\sqrt{2n\log n\log\log\log n}}=0\ \ \text{a.s.}
\end{equation}
Finally, on substituting \eqref{Nn} in \eqref{N_nas} and then using \eqref{asymc_n} and \eqref{M_nsup}, we obtain the required result.
\end{proof}
\begin{remark}
On substituting $a=1$ in \eqref{Snloglog}, we obtain
\begin{equation*}	
	\limsup_{n\to\infty}\frac{S_n}{\sqrt{2n\log n\log\log\log n}}=1\ \ \text{a.s.},
\end{equation*}
which is Eq. (3.8) of Bercu (2018).
\end{remark}
\begin{theorem}
Let $p=(4a-1)/4a$. Then,
\begin{equation}\label{Snnorcrit}
	\frac{S_n}{\sqrt{n\log n}}\xrightarrow{d}\mathcal{N}(0,(2a-1)^2).
\end{equation}
\end{theorem}
\begin{proof}	
Let $\{B_n\}_{n\ge a}$ be a sequence of positive definite diagonal matrices given by
\begin{equation*}
	B_n=\frac{1}{\sqrt{n\log n}}\begin{pmatrix}
		b_n^{-1}c_n & 0\vspace{.2cm}\\
		0 & 1
	\end{pmatrix}.
\end{equation*}
Then, from \eqref{qk}, we get
{\footnotesize\begin{equation*}
B_n\Big(\sum_{k=a}^{n}Q_k\Big)B_n^t=\frac{1}{n\log n}\begin{pmatrix}
	(a(2p-1))^2b_n^{-1}c_n\sum_{k=a}^{n}(b_{k}c_{k}^{-1})^2 & a(2p-1)(a-1)b_n^{-1}c_n\sum_{k=a}^{n}b_{k}c_{k}^{-1}\vspace{.2cm}\\
	a(2p-1)(a-1)b_n^{-1}c_n\sum_{k=a}^{n}b_{k}c_{k}^{-1} & (a-1)^2(n-a+1)
\end{pmatrix}.
\end{equation*}}
From \eqref{asymb_n} and \eqref{asymc_n}, we have
\begin{equation*}
	b_n^{-1}c_n\sum_{k=a}^{n}b_{k}c_{k}^{-1}\sim 2n 
\end{equation*}
and 
\begin{equation*}
B_n^tB_n\sim\begin{pmatrix}
	\Big(\frac{2a-1}{\Gamma(a+1/2)}\Big)^2\frac{1}{\log n} & 0\vspace{.2cm}\\
	0 & \frac{1}{n\log n}
\end{pmatrix}.
\end{equation*}
Note that $\lambda_{\max}(B_n^tB_n) \to 0$ as $n\to \infty$. So, $\|B_n\|\to 0$ as $n\to \infty$.
Also,
{\small\begin{align*}
B_n&\Big(\sum_{k=a}^{n-1}Q_{k+1}\Big(\frac{c_{k+1}Y_k}{k-a+1}\Big)^2\Big)B_n^t\\
&=\ \ \begin{pmatrix}
	\frac{(a(2p-1)b_{n}^{-1}c_{n})^2}{n\log n}\sum_{k=a}^{n-1}\Big(\frac{b_{k+1}Y_k}{k-a+1}\Big)^2 & \frac{a(2p-1)(a-1)b_{n}^{-1}c_{n}}{n\log n}\sum_{k=a}^{n-1}b_{k+1}c_{k+1}\Big(\frac{Y_k}{k-a+1}\Big)^2\vspace{.2cm}\\
	\frac{a(2p-1)(a-1)b_{n}^{-1}c_{n}}{n\log n}\sum_{k=a}^{n-1}b_{k+1}c_{k+1}\Big(\frac{Y_k}{k-a+1}\Big)^2 & \frac{(a-1)^2}{n\log n}\sum_{k=a}^{n-1}\Big(\frac{c_{k+1}Y_k}{k-a+1}\Big)^2
\end{pmatrix}
\end{align*}}
is a matrix whose each entry goes to zero as $n\to\infty$. Thus, from \eqref{predquad}, it follows that
\begin{equation*}
	\lim_{n\to\infty}B_n\left\langle\mathcal{M}\right\rangle_{n}B_n^t=B\ \ \text{a.s.},
\end{equation*}
where $B$ is a positive semi-definite matrix given by
\begin{equation*}
	B=\begin{pmatrix}
		(2a-1)^2 & 0\vspace{.2cm}\\
		0 & 0
	\end{pmatrix}.
\end{equation*}
From this step, the proof follows similar lines to that of Theorem \ref{thmasynrdiff}, and therefore the remaining steps are omitted. This completes the proof.
\end{proof}
\begin{remark}
On substituting $a=1$ in \eqref{Snnorcrit}, we get $S_n/\sqrt{n\log n}\xrightarrow{d}\mathcal{N}(0,1)$,
which is Eq. (3.10) of Bercu (2018).
\end{remark}

\begin{theorem}
Let $p=(4a-1)/4a$. Then, $\mathbb{E}(S_n^2)\sim (2a-1)^2n\log n$ as $n\to\infty$.
\end{theorem}
\begin{proof}
Recall that
\begin{equation*}
B_n=\frac{1}{\sqrt{n\log n}}\begin{pmatrix}
		b_n^{-1}c_n & 0\vspace{.2cm}\\
		0 & 1
	\end{pmatrix}.
\end{equation*}
Also,
\begin{equation*}
\frac{S_n^2}{n\log n}=v^tB_n\mathcal{M}_n(v^tB_n\mathcal{M}_n)^t,
\end{equation*}
where $v=\begin{pmatrix}
	-1 & 1
\end{pmatrix}^t$.
Thus,
\begin{equation}\label{Snacrit}
\mathbb{E}\Big(\frac{S_n^2}{n\log n}\Big)=\frac{1}{n\log n}((b_n^{-1}c_n)^2\mathbb{E}(M_n^2)-2b_n^{-1}c_n\mathbb{E}(M_nN_n)+\mathbb{E}(N_n^2)),
\end{equation}
which follows from \eqref{twotemmar}. 
By using \eqref{asymb_n}, \eqref{asymc_n} and Lemma \ref{Yn2asym} in \eqref{predquad}, we get
\begin{align}\label{threasymcrit}
\left.\begin{aligned}
	\mathbb{E}(M_n^2)&\sim \Big(\Gamma\Big(a+\frac{1}{2}\Big)\Big)^2\log n,\\
	\mathbb{E}(M_nN_n)&\sim 4(1-a)\Gamma\Big(a+\frac{1}{2}\Big)\sqrt{n},\\
	\mathbb{E}(N_n^2)&\sim 4(a-1)^2n.
\end{aligned}\right\}
\end{align}
Finally, the required result is obtained from \eqref{Snacrit} by using \eqref{asymb_n}, \eqref{asymc_n} and \eqref{threasymcrit}.
\end{proof}
\subsubsection{The superdiffusive regime} Here, we discuss the results obtained for the ERW with delayed amnesia in superdiffusive regime.
\begin{theorem}\label{thmsuperdiff}
Let $(4a-1)/4a<p\le 1$. Then,
\begin{equation}\label{supdefcon}
\lim_{n\to\infty}\frac{S_n}{n^{2a(p-1)+1}}=L \ \ \text{a.s.,}
\end{equation}
where $L$ is some non-degenerate random variable. Also, we have the following mean square convergence:
\begin{equation}\label{supdefsq}
\lim_{n\to\infty}\mathbb{E}\Big(\Big|\frac{S_n}{n^{2a(p-1)+1}}-L\Big|^2\Big)=0.
\end{equation}
\end{theorem}
\begin{proof}
By using \eqref{DelMnEps} in \eqref{predMn}, we get
\begin{align*}
\langle M\rangle_n&=\Big(\frac{a(2p-1)}{2a(1-p)-1}\Big)^2\sum_{k=a}^{n}(b_kc_k^{-1})^2\mathbb{E}(\varepsilon_k^2|\mathcal{F}_{k-1})\\
&\le \Big(\frac{a(2p-1)}{2a(1-p)-1}\Big)^2\sum_{k=a}^{n}(b_kc_k^{-1})^2,
\end{align*}
which follows from \eqref{epsPow2}. Since $\{\langle M\rangle_n\}_{n\ge a}$ is an increasing sequence, by using Proposition \ref{asymps_n}(iii) and Theorem \ref{thm_LLN1}, we have
$\lim_{n\to\infty}b_nY_n=M$ a.s. Consequently, by using \eqref{asymb_n}, we get
\begin{equation}\label{supdef11}
\lim_{n\to\infty}\frac{Y_n}{n^{a(2p-1)}}=\frac{M}{\Gamma(a(2p-1)+1)}\ \ \text{a.s.},
\end{equation}
where
\begin{equation*}
M=\frac{a(2p-1)}{2a(1-p)-1}\sum_{k=a}^{\infty}b_kc_k^{-1} \varepsilon_k.
\end{equation*} 
 Now, by using \eqref{Nn} in Lemma \ref{Nn_O}, we get
\begin{equation}\label{Snsup}
\Big(\frac{S_n+c_nY_n}{n^{2a(p-1)+1}}\Big)^2=O\Big(\frac{(\log n)^{1+\gamma}}{n^{4a(p-1)+1}}\Big)\ \ \text{a.s.}
\end{equation}
Finally, by using \eqref{asymc_n} and \eqref{supdef11} in \eqref{Snsup}, we obtain the required result \eqref{supdefcon} with
\begin{equation}\label{Lvariab}
L=\frac{a(1-2p)M}{(2a(1-p)-1)\Gamma(a(2p-1)+1)}.
\end{equation}
Now, to prove the mean square convergence, we consider
\begin{equation*}
\mathbb{E}(\langle M\rangle_n)=\sum_{k=a}^{n}\mathbb{E}((\Delta M_k)^2)=\mathbb{E}(M_n^2),
\end{equation*}
which follows on using $M_{a-1}=0$. Since $\langle M\rangle_{n}\le s_n$, Proposition \ref{asymps_n}(iii) implies that $\sup_{n\ge a}\mathbb{E}(M_n^2)<\infty$. So, $\{M_n,\mathcal{F}_n\}_{n\ge a}$ is $\mathbb{L}^2$-bounded. Therefore,
\begin{equation}\label{supdef1}
\lim_{n\to\infty}\mathbb{E}(|M_n-M|^2)=0. 
\end{equation}
Also, by using \eqref{epsPow2} in \eqref{DelNn}, we get
\begin{equation*}
\langle N\rangle_n\le \Big(\frac{a(2p-1)}{2a(1-p)-1}\Big)^2(n-a+1).
\end{equation*}
 Since, $N_{a-1}=S_{a-1}$, we have $
\mathbb{E}(N_n^2)\le \big(\frac{a(2p-1)}{2a(1-p)-1}\big)^2(n-a+1)+\mathbb{E}(S_{a-1}^2)$. Thus,
\begin{equation}\label{supdef2}
\lim_{n\to\infty}\mathbb{E}\Big(\Big(\frac{N_n}{n^{2a(p-1)+1}}\Big)^2\Big)=0.
\end{equation}
Finally, \eqref{supdefsq} follows from \eqref{supdef1} and \eqref{supdef2}.
\end{proof}

\begin{theorem}
Let $(4a-1)/4a<p\le 1$. Then, $\mathbb{E}(S_n^2)\sim \alpha n^{4a(p-1)+2}$ as $n\to\infty$, where $\alpha$ is some constant.
\end{theorem}
\begin{proof}
By using \eqref{asymb_n}, \eqref{asymc_n}, Proposition \ref{asymps_n}(iii) and Lemma \ref{Yn2asym}, the proof follows similar lines to that of Theorem \ref{meansqdiff}. 
\end{proof}
\subsection{Center of mass}Here, we give the almost sure convergence results related to the center of mass of ERW with delayed amnesia. It is defined as follows:
\begin{equation*}
G_n\coloneqq\frac{1}{n}\sum_{k=1}^{n}S_k.
\end{equation*}

The proofs of the following results are analogous to that of the corresponding results in Chen and Laulin (2023), and are therefore omitted.
\begin{theorem}
The following almost sure convergence results hold true:\\
\noindent{(i)} For $0\le p< (4a-1)/4a$, we have
\begin{equation*}
	\lim_{n\to\infty}\frac{G_n}{n}=0\ \ \text{a.s.},
\end{equation*}
\noindent{(ii)} for $p= (4a-1)/4a$, we have
\begin{equation*}
	\lim_{n\to\infty}\frac{G_n}{\sqrt{n}\log n}=0\ \ \text{a.s., and}
\end{equation*}
\noindent{(iii)} for $(4a-1)/4a<p\le 1$, we have
\begin{equation*}
\lim_{n\to\infty}\frac{G_n}{n^{2a(p-1)+1}}=\frac{L}{(2a(p-1)+2)}\ \ \text{a.s.},
\end{equation*}
where $L$ is as given in \eqref{Lvariab}. 
\end{theorem}

\section{ERW with delayed amnesia and random step sizes}\label{randmstep}
Here, we study an extension of the ERW with delayed amnesia by considering random step sizes. We call it the ERW with delayed amnesia and random step sizes, and denote it by $\{W_n\}_{n\ge1}$. It is defined as follows: $W_n\coloneqq T_1+T_2+\dots+T_n$, where $T_k=X_kZ_k$ for all $k=1,2,\dots,n$. Here, $\{Z_n\}_{n\ge1}$ is a sequence of step sizes of the walk which are iid positive random variables. Let $Z_1$ has finite mean $\mu$ and finite variance $\sigma^2$. We assume that $Z_k$'s are independent of $X_k$'s. Then, $\{\mathscr{M}_n,\mathcal{G}_n\}_{n\ge1}$ is a zero mean martingale, where $\mathscr{M}_n\coloneqq W_n-\mu S_n$ and $\mathcal{G}_n=\sigma(T_1, T_2,\dots, T_n)$.

For $n\ge 1$, let $\delta_{n}=Z_{n}-\mu$. Also, let us consider the following martingale differences:
\begin{equation}\label{Delmathscr}
\Delta \mathscr{M}_{n}=\mathscr{M}_{n}-\mathscr{M}_{n-1}=X_{n}\delta_{n},
\end{equation}
with $\mathscr{M}_0=0$. Then, $\mathbb{E}(\delta_{n+1}|\mathcal{G}_n)=0$ and
$\sup_{n\ge1}\mathbb{E}(\delta_{n+1}^2|\mathcal{G}_n)=\sigma^2<\infty$.

Note that $\mathscr{M}_n=X_1\delta_1+X_2\delta_2+\dots+X_n\delta_n$. Let $v_n=X_1^2+X_2^2+\dots+X_n^2=n$. Then, $\lim_{n\to \infty}v_n=\infty$. Thus, from Theorem \ref{thm_LLN2}, we get
\begin{equation}\label{W_ncong}
\lim_{n\to\infty}\frac{W_n-\mu S_n}{n}=0\ \ \text{a.s.}
\end{equation}
By using Theorem \ref{thmdiffu}, Theorem \ref{thmcirit} and Theorem \ref{thmsuperdiff} in \eqref{W_ncong}, we get the following result:
\begin{theorem}
Let $0\le p<1$. Then, $\lim_{n\to \infty}W_n/n=0$ a.s.
\end{theorem}
 
The predictable quadratic variation of $\{\mathscr{M}_n\}_{n\ge 1}$ is
\begin{equation}\label{quadvar}
\langle \mathscr{M}\rangle_n=\sum_{k=1}^{n}\mathbb{E}((\Delta \mathscr{M}_k)^2|\mathcal{G}_{k-1})=n\sigma^2.
\end{equation}
So, $\lim_{n\to\infty}\langle \mathscr{M}\rangle_n=\infty$ a.s. Let $y_n=\sqrt{2\log \log \langle \mathscr{M}\rangle_n}$ and the sequence $\{Z_n\}_{n\ge 1}$ be uniformly bounded. Then, 
\begin{equation*}
|\Delta \mathscr{M}_n|\le L_n\frac{\sqrt{\langle \mathscr{M}\rangle_n}}{y_n}.
\end{equation*}
Here, $L_n=by_n/\sqrt{\langle \mathscr{M}\rangle_n}$, where $b$ is some constant. 

Note that $L_n$ is $\mathcal{G}_{n-1}$-measurable and $\lim_{n\to \infty}L_n=0$ a.s. Thus, from Theorem \ref{stoutthm2}, we get
\begin{equation*}
\limsup_{n\to\infty}\frac{\mathscr{M}_n}{y_n\sqrt{\langle \mathscr{M}\rangle_n}}=1\ \ \text{a.s.}
\end{equation*}
Therefore, 
\begin{equation}\label{lilWn}
\limsup_{n\to\infty}\frac{\mathscr{M}_n}{\sqrt{2n\log\log n}}=|\sigma|\ \ \text{a.s.}
\end{equation}

Now, by using Theorem \ref{critlil} in \eqref{lilWn}, we get the following result:
\begin{theorem}
Let $p=(4a-1)/4a$ and $\{Z_n\}_{n\ge 1}$ be uniformly bounded. Then, we have the following law of iterated logarithm for $\{W_n\}_{n\ge1}$:
\begin{equation*}
\limsup_{n\to\infty}\frac{W_n}{\sqrt{2n\log n\log \log\log n}}=\mu (2a-1)\ \ \text{a.s.}
\end{equation*}
\end{theorem}

The next result is obtained by using Theorem \ref{thmsuperdiff} in \eqref{lilWn}.
\begin{theorem}
Let $(4a-1)/4a<p\le1$. Then,
	\begin{equation*}
		\lim_{n\to\infty}\frac{W_n}{n^{2a(p-1)+1}}=\mu L\ \ \text{a.s.},
	\end{equation*}
where $L$ is as given in \eqref{Lvariab}.
\end{theorem}

\begin{lemma}\label{cltmathscr}
Let $\{Z_n\}_{n\ge1}$ be uniformly bounded. Then, $\mathscr{M}_n/\sqrt{n}\xrightarrow{d}\mathcal{N}(0,\sigma^2)$.
\end{lemma}
\begin{proof}
From \eqref{Delmathscr} and \eqref{quadvar}, we get
\begin{equation}\label{cltcondi2}
	\lim_{n\to \infty}\frac{1}{n}\sum_{k=1}^{n}\mathbb{E}((\Delta\mathscr{M}_k)^2 \mathbb{I}_{\{|\Delta\mathscr{M}_k| \ge \varepsilon \sqrt{n}\}}|\mathcal{G}_{k-1})=0\ \ \text{a.s.},
\end{equation}
for all $\varepsilon>0$ and
\begin{equation}\label{cltcondi1}
	\frac{\langle\mathscr{M}\rangle_{n}}{n}=\sigma^2,
\end{equation}
respectively. The required result follows from \eqref{cltcondi2}, \eqref{cltcondi1} and Theorem \ref{CLT}.
\end{proof}

\begin{theorem}
Let $0\le p<(4a-1)/4a$ and $\{Z_n\}_{n\ge1}$ be uniformly bounded. Then,
\begin{equation}\label{normrandmstep}
\frac{W_n}{\sqrt{n}}\xrightarrow{d}\mathcal{N}\Big(0,\sigma^2+\frac{\mu^2K}{(2a(1-p)-1)^2}\Big),
\end{equation}
where $K$ is as given in \eqref{K}.
\end{theorem}
\begin{proof}
The joint characteristic function of $\big(\mathscr{M}_n/\sqrt{n},\mu S_n/\sqrt{n}\big)$ is given by
\begin{equation}\label{CF}
\phi_n(s,t)=\mathbb{E}\Big(\exp\Big\{i\Big(\frac{s\mathscr{M}_n}{\sqrt{n}},\frac{t\mu S_n}{\sqrt{n}}\Big)\Big\}\Big), \ \ s,\,t\in\mathbb{R}.
\end{equation}
Let $\mathcal{F}=\sigma(\{X_k:k\ge 1\})$. Then, conditional on $\mathcal{F}$, $\{X_n(Z_n-\mu)\}_{n\ge 1}$ is a sequence of independent random variables. Therefore, by the classical CLT for independent random variables, we obtain
\begin{equation}\label{fordct}
	\lim_{n\to \infty}\mathbb{E}\Big(\exp\Big\{is\frac{\mathscr{M}_n}{\sqrt{n}}\Big\}\Big|\mathcal{F}\Big)=\exp\Big\{-\frac{1}{2}s^2\sigma^2\Big\}\ \ \text{a.s.}
\end{equation}
From \eqref{CF}, we have
\begin{align}\label{CF2}
\phi_n(s,t)&=f_n+\exp\Big\{-\frac{1}{2}s^2\sigma^2\Big\}\mathbb{E}\Big(\exp\Big\{it\frac{\mu S_n}{\sqrt{n}}\Big\}\Big),
\end{align}
where
\begin{equation*}
f_n=\mathbb{E}\Big(\Big(\mathbb{E}\Big(\exp\Big\{is\frac{\mathscr{M}_n}{\sqrt{n}}\Big\}\Big|\mathcal{F}\Big)-\exp\Big\{-\frac{1}{2}s^2\sigma^2\Big\}\Big)\exp\Big\{it\frac{\mu S_n}{\sqrt{n}}\Big\}\Big).
\end{equation*} 
Note that
\begin{equation}\label{modVn}
|f_n|\le \mathbb{E}\Big(\Big|\mathbb{E}\Big(\exp\Big\{is\frac{\mathscr{M}_n}{\sqrt{n}}\Big\}\Big|\mathcal{F}\Big)-\exp\Big\{-\frac{1}{2}s^2\sigma^2\Big\}\Big|\Big).
\end{equation}
By using \eqref{fordct} and the dominated convergence theorem in \eqref{modVn}, we obtain $\lim_{n\to\infty}f_n=0$.
Thus, from \eqref{CF2}, we have
\begin{equation}\label{diffcf}
\lim_{n\to\infty}\phi_n(s,t)=\exp\Big\{-\frac{1}{2}s^2\sigma^2\Big\}\lim_{n\to\infty}\mathbb{E}\Big(\exp\Big\{it\frac{\mu S_n}{\sqrt{n}}\Big\}\Big).
\end{equation}
By using \eqref{Snnordiff} in \eqref{diffcf}, we get
\begin{equation*}
\lim_{n\to\infty}\phi_n(s,s)=\exp\Big\{-\frac{1}{2}s^2\Big(\sigma^2+\frac{\mu^2K}{(2a(1-p)-1)^2}\Big)\Big\},
\end{equation*}
where $K$ is as given in \eqref{K}. This completes the proof.
\end{proof}

\begin{remark}
On substituting $\mu=a=1$ in \eqref{normrandmstep}, we get $W_n/\sqrt{n}\xrightarrow{d}\mathcal{N}(0,\sigma^2+1/(3-4p))$, which is the corresponding result for ERW with random step sizes (see Fan and Shao (2024), p. 2052).
\end{remark}

\begin{theorem}
Let $p=(4a-1)/4a$ and $\{Z_n\}_{n\ge1}$ be uniformly bounded. Then,
\begin{equation*}
	\frac{W_n}{\sqrt{n\log n}}\xrightarrow{d}\mathcal{N}\big(0,(\mu(2a-1))^2\big).
\end{equation*}
\end{theorem}
\begin{proof}
By using Slutsky's theorem in Lemma \ref{cltmathscr}, we get
\begin{equation}\label{critMn}
	\frac{\mathscr{M}_n}{\sqrt{n\log n}}\xrightarrow{d}0.
\end{equation}
Again, on applying Slutsky's theorem in \eqref{Snnorcrit} and \eqref{critMn}, we obtain the required result.
\end{proof}

\section*{Acknowledgement}
The first author thanks Government of India for the grant of Prime Minister's Research Fellowship, ID 1003066.

\end{document}